\documentclass[11pt]{article}
\usepackage[utf8]{inputenc}
\usepackage{comment}
\usepackage{amssymb,amsmath,amsthm,graphicx,xcolor,mathtools,enumerate,mathrsfs}
\usepackage{enumitem,thmtools,refcount}
\usepackage{thm-restate}
\usepackage[noadjust]{cite}
\usepackage{booktabs}
\usepackage{float}
\usepackage{subcaption}
\usepackage{hyperref}
\usepackage[export]{adjustbox}
\hypersetup{
  colorlinks=true,
  linkcolor=blue,   
  citecolor=green,   
  urlcolor=magenta    
}
\usepackage{cleveref}
 \usepackage[margin=2.54cm]{geometry}
\usepackage{amsthm}
\usepackage{placeins}
\usepackage{sectsty}
\usepackage{array}
\usepackage{multirow}
\sectionfont{\centering}
\crefrangelabelformat{lemma}{#3#1#4-#5#2#6}

\theoremstyle{plain}
\newtheorem{theorem}{Theorem}[section]
\crefname{theorem}{Theorem}{Theorems}

\newtheorem{proposition}[theorem]{Proposition}
\crefname{proposition}{Proposition}{Propositions}

\newtheorem{corollary}[theorem]{Corollary}
\crefname{corollary}{Corollary}{Corollaries}

\newtheorem{lemma}[theorem]{Lemma}
\crefname{lemma}{Lemma}{Lemmas}

\newtheorem{conjecture}[theorem]{Conjecture}
\crefname{conjecture}{Conjecture}{Conjectures}

\crefname{problem}{Problem}{Problem}

\newtheorem{claim}[theorem]{Claim}
\crefname{claim}{Claim}{Claims}

\crefname{observation}{Observation}{Observations}

\crefname{setup}{Setup}{Setups}

\crefname{fact}{Fact}{Facts}

\crefname{remark}{Remark}{Remarks}

\crefname{example}{Example}{Examples}

\theoremstyle{definition}
\newtheorem{definition}[theorem]{Definition}
\crefname{definition}{Definition}{Definitions}

\crefname{construction}{Construction}{Constructions}

\crefname{question}{Question}{Questions}

\crefname{section}{Section}{Sections}
\Crefname{section}{Section}{Sections}

\crefname{subsection}{Subsection}{Subsections}
\Crefname{subsection}{Subsection}{Subsections}

\numberwithin{equation}{section}

\allowdisplaybreaks

\definecolor{DarkDesaturatedBlue}{HTML}{3A3556}
\definecolor{VividOrange}{HTML}{F15918}
\definecolor{PureOrange}{HTML}{FFBA00}
\definecolor{LightGrayishPink}{HTML}{EEC5D5}
\definecolor{VerySoftBlue}{HTML}{B5AFDB}

\definecolor{DarkDesaturatedBlue}{HTML}{3A3556}
\definecolor{VividOrange}{HTML}{F15918}
\definecolor{PureOrange}{HTML}{FFBA00}
\definecolor{LightGrayishPink}{HTML}{EEC5D5}
\definecolor{VerySoftBlue}{HTML}{B5AFDB}
  \makeatletter
  \newcommand{\labelinthm}[1]{%
     \label{temp#1}
     \protected@write \@auxout {}{\string \newlabel{#1}{{\emph{\ref{temp#1}}}{\thepage}{\emph{\ref{temp#1}}}{temp#1}{}} }%
  }
  \makeatother

\title{An Ore-type Theorem for Oriented Discrepancy of Hamilton Cycles}
\author{Yufei Chang\thanks{School of Mathematics, Shandong University, Jinan 250100, China. Supported by National Natural Science Foundation of China (No.12571373).}
\and Yangyang Cheng\thanks{Faculty of Computer Science and Mathematics, University of Passau, Passau, Germany. Partially supported by the Deutsche Forschungsgemeinschaft (DFG, German Research Foundation)-542321564.}
\and Zhilan Wang\footnotemark[1] \thanks{Corresponding author. Email: {\tt 202690000074@sdu.edu.cn}}
\and Shuo Wei\footnotemark[1]
\and Jin Yan\footnotemark[1]}

\providecommand{\keywords}[1]{\textbf{Keywords:} #1}
\providecommand{\subjclass}[1]{\textbf{Mathematics Subject Classifications:} #1}

\begin{document}

\date{}
\maketitle

\begin{abstract}
Oriented discrepancy concerns the existence of prescribed subgraphs whose edges exhibit a substantial imbalance between the two directions. The Hamilton cycle case was first investigated by Gishboliner, Krivelevich, and Michaeli. Subsequently, Freschi and Lo established a sharp minimum-degree result and raised the question of whether a corresponding Ore-type theorem holds. For an oriented graph \(G\) on \(n\) vertices, define $\sigma_2(G):=\min\{d_G(x)+d_G(y):x\neq y,\ xy,yx\notin E(G)\},$ with \(\sigma_2(G):=2(n-1)\) when \(G\) is a tournament. Given a cycle \(C\) in an oriented graph, let \(\sigma_{\max}(C)\) denote the maximum number of edges oriented consistently with one of the two traversal directions of \(C\). We prove that for every \(\gamma>0\) there exists \(n_0\) such that every oriented graph \(G\) on \(n\ge n_0\) vertices with \(\sigma_2(G)\ge n\) contains a Hamilton cycle \(C\) satisfying $$ \sigma_{\max}(C)\ge \max\left\{\frac n2,\frac{\sigma_2(G)}2-\gamma n\right\}. $$ This partially answers a question of Freschi and Lo and gives an asymptotic affirmative answer to a conjecture of Ai et al. The bound is asymptotically best possible.
\end{abstract}



\section{Introduction}\label{sec:intro}

Graph discrepancy, introduced by Erd\H{o}s and Spencer  \cite{Erdos1963,ErdosSpencer1972}, extends classical discrepancy theory by studying the unavoidable imbalance in two-colorings of edges. 
Erd\H{o}s, F\"{u}redi, Loebl, and S\'{o}s \cite{ErdosFurediLoeblSos1995} showed that there is a  constant $c>0$ such that any 2-edge coloring of $K_n$ admits, for every $n$-vertex tree $T$ of maximum degree $\Delta$, a copy of $T$ with at least $(n-1)/2+c(n-1-\Delta)$ edges of the same color. This motivates the broader problem of determining conditions on a 2-edge-colored graph $G$ that guarantee a copy of a given subgraph $H$ with large discrepancy.
The investigation of Hamilton cycle discrepancy in dense graphs can be traced back to the work of Balogh, Csaba, Jing, and Pluh\'ar \cite{Bal}.

\begin{theorem}[\cite{Bal}]\label{Bal}
Let $0 < c < 1/4$ and $n \in \mathbb{N}$ be sufficiently large. If $G$ is a $2$-edge-colored graph on $n$ vertices with $\delta(G) \geq (3/4 + c)n$, then $G$ contains a Hamilton cycle with at least $(1/2 + c/64)n$ edges in one color.
\end{theorem}

The study of this problem has developed in several directions over the past few years. Results concerning perfect matchings and Hamilton cycles can be found in \cite{Bal,Fre,Gishboliner,Gis,Gis1,Gis2}, while \cite{Gis1}, \cite{BaloghC,Bradac2} and \cite{Bradac1} deal with spanning trees, $H$-factors and powers of Hamilton cycles, respectively.

Gishboliner, Krivelevich and Michaeli \cite{Gis2} introduced the oriented discrepancy problem for Hamilton cycles, in which the two colors are replaced by the two possible directions relative to a traversal of the cycle.
We consider an oriented graph $G$, namely, a directed graph without loops in which each pair of distinct vertices is joined by at most one edge. For each $v\in V(G)$, let $d_G^+(v)$ and $d_G^-(v)$ denote its outdegree and indegree, respectively. We define $\delta(G)=\min_{v\in V(G)}\bigl(d_G^+(v)+d_G^-(v)\bigr),$
which is called the \textit{minimum total degree} of $G$, or simply the minimum degree when the context is clear.

For an \(n\)-vertex oriented graph \(G\) and an oriented cycle \(C=v_1\ldots v_\ell v_1\), 
define
\[
\sigma^+(C):=|\{1\le i\le\ell: v_iv_{i+1}\in E(G)\}|,\qquad
\sigma^-(C):=|\{1\le i\le\ell: v_{i+1}v_i\in E(G)\}|,
\]
which denote the numbers of forward and backward edges of \(C\), respectively, where the indices are taken modulo \(\ell\).
Let \(\sigma_{\max}(C):=\max\{\sigma^+(C),\sigma^-(C)\}\) be  the number of edges oriented in the \emph{dominant direction}, and \(\sigma_{\min}(C):=\min\{\sigma^+(C),\sigma^-(C)\}\) be the number of edges in the \emph{minority direction}. The analogous notions are defined for oriented paths.

Motivated by \cref{Bal}, Gishboliner, Krivelevich, and Michaeli \cite{Gis2} conjectured a generalization of Dirac's theorem. This conjecture was subsequently confirmed by Freschi and Lo \cite{Freschi}.

\begin{theorem}[\cite{Freschi}]\label{Lo}
Let $G$ be an oriented graph on $n \geq 3$ vertices. If $\delta(G) \geq \frac{n}{2}$ then there exists a Hamilton cycle $C$ in $G$ such that $\sigma_{\max}(C) \geq \delta(G)$.
\end{theorem}

Freschi and Lo \cite{Freschi} asked whether their theorem admits an Ore-type analogue. For a graph or oriented graph \(G\), define \(\sigma_2(G):=\min_{x,y\in V(G)}\{d(x)+d(y)\mid x\ne y,xy,yx\notin E(G)\}\), and for a complete graph or tournament set \(\sigma_2(G):=2(n-1)\). Ai et al. \cite{Guo} subsequently formulated the following precise conjecture.

\begin{conjecture}[\cite{Guo}] \label{pro1}
Let $G$ be an oriented graph on $n \ge 3$ vertices. If $\sigma_2(G) \ge n$, then $G$ contains a Hamilton cycle $C$ such that $\sigma_{\max}(C) \ge \sigma_2(G)/2$.
\end{conjecture}
Ai et al. \cite{Guo} also gave constructions showing that, under the hypothesis \(\sigma_2(G)\ge n\), one cannot guarantee \(\sigma_{\max}(C)>\sigma_2(G)/2\). Thus the conjectured bound is best possible.
Our main theorem confirms an asymptotic version of this conjecture when \(n\) is sufficiently large.
Thus, \cref{th2} can be viewed as an asymptotic oriented-discrepancy analogue  of Ore's theorem \cite{Ore}.

\begin{theorem} \label{th2}
For every $\gamma > 0$, there exists a positive integer $n_0$ such that the following holds.
Let $G$ be an oriented graph on $n \ge n_0$ vertices with $\sigma_2(G)\ge n$.
Then $G$ contains  a Hamilton cycle $C$ such that  \[\sigma_{\max}(C) \ge  \max\left\{\frac{n}{2},\frac{\sigma_2(G)}{2}-\gamma n\right\}.\]
\end{theorem} 

Up to an additive $o(n)$ term, \cref{th2} asymptotically generalizes \cref{Lo} to the Ore-type setting.
Applying \cref{th2} with $\gamma=\eta/4$,
we obtain the following corollary.

\begin{corollary}
For every $\eta>0$, there exist $0<\xi\le\eta/4$
and $n_0\in\mathbb N$ such that every oriented graph
$G$ on $n\ge n_0$ vertices with
$\sigma_2(G)\ge(1+\eta)n$ contains a Hamilton
cycle $C$ satisfying $\sigma_{\max}(C)\ge(1/2+\xi)n$.
\end{corollary}

\smallskip
\noindent\textbf{Organization:}
The paper is structured as follows.  \cref{not,ov} cover the notation and outline the proof of \cref{th2}. In \cref{sec:tool}, we collect several auxiliary results and tools, including the Regularity Lemma. The absorbing, reservoir, and almost covering lemmas are developed in \cref{sec:ore}, followed by the proof of \cref{th2} in \cref{sec:proof1}. We conclude with some remarks in \cref{sec:remark}.

\section{Preliminaries and proof overview}\label{sec:prep}

\subsection{Notation}\label{not}
For a digraph \(G\), let \(U(G)\) denote its underlying graph, obtained from \( G \) by ignoring the directions of its edges. An oriented graph is a digraph without loops such that between every two vertices there is at most one edge.
The number of vertices \( |G| := |V(G)| \) is called the \textit{order} of \( G \), and the number of edges \( e(G) := |E(G)| \) is called the \textit{size} of \( G \). Given two vertices \( x \) and \( y \) of \( G \), we write \( xy \) for the edge directed from \( x \) to \( y \). 
We write $N_G^+(x)$ and $N_G^-(x)$ for the out-neighborhood and in-neighborhood of a vertex $x$ in $G$, and $d_G^+(x):=|N_G^+(x)|$, $d_G^-(x):=|N_G^-(x)|$ for its outdegree and indegree. We write $N_G(x):=N_G^+(x)\cup N_G^-(x)$ for the neighborhood of $x$ and $d_G(x):=|N_G(x)|$ for its total degree.

For any $X\subseteq V(G)$, let $G[X]$ denote the subdigraph induced by $X$, and set $G\setminus X=G[V(G)\setminus X]$. For \(k \in \mathbb{N}\), we write \([k] = \{1, 2, \dots, k\}\).

A \emph{directed path} (resp.\ \emph{directed cycle}) is a path $v_1v_2\cdots v_\ell$ (resp.\ a cycle $v_1v_2\cdots v_\ell v_1$) such that $v_iv_{i+1}\in E(G)$ for every $1\le i<\ell$ (resp.\ for every $1\le i\le\ell$, with indices taken modulo $\ell$). A path or a cycle is called \emph{oriented} if every one of its edges carries a direction. 
An oriented path (respectively, cycle) in \(G\) is a path (respectively, cycle) of the underlying graph \(U(G)\), with each of its edges carrying the orientation inherited from \(G\). Such a path or cycle need not be directed. Throughout this paper, paths and cycles are assumed to be oriented unless specified otherwise.
An $(x,y)$-path is a path with end vertices $x$ and $y$.
The length of a path or a cycle is defined as the number of its edges. A \emph{segment} of a path $P$ is a subpath consisting of consecutive vertices of $P$. 

Given a (di)graph $G$ and $t\in \mathbb{N}$, we let $G(t)$ denote the $t$-blow-up of $G$. More precisely, $V(G(t)):=\{v^j : v\in V(G)\text{ and } j\in [t]\}$ and $E(G(t)):=\{v^m w^\ell : vw\in E(G)\text{ and } m,\ell\in [t]\}$.  
A tournament $T_n$ is an orientation of $K_n$.

For a family $\mathcal F$ of pairwise vertex-disjoint oriented paths or cycles, we write $\sigma_{\max}(\mathcal F):=\sum_{F\in\mathcal F}\sigma_{\max}(F).$

Throughout the paper, we omit all floor and ceiling functions whenever these are not crucial. The constants in the hierarchies used to state our results are chosen from right to left. For example, if we claim that a result holds whenever \( 0 < a \ll b \ll c \leq 1 \), then there are non-decreasing functions  $f : (0, 1] \to (0, 1] \text{ and } g : (0, 1] \to (0, 1]$
such that the result holds for all \( 0 < a, b, c \leq 1 \) with \( b \leq f(c) \) and \( a \leq g(b) \).

\smallskip

\subsection{Proof Strategy}\label{ov}

   The proof of \cref{th2} is based on the absorbing method, introduced by R\"{o}dl, Ruci\'{n}ski and Szemer\'{e}di \cite{rodl}. 
   Specifically, our absorption method adopts the two-step absorption technique developed in \cite{Chang}. 

Suppose that $0<\varepsilon\ll \mu \ll\eta\ll 1$. If \( \sigma_2(G) \) is close to \( n \), then   \cref{th2} follows immediately from Ore's theorem. We may therefore assume that \( \sigma_2(G)\ge (1+\eta)n \), and in this case the proof proceeds in four steps:

\smallskip

\textbf{$\bullet$ Step 1: Absorbing structure.} 
We first construct small absorbing structures  for vertices. We next use a connecting lemma to join the chosen absorbers into a path \(P_{abs}\). The path \(P_{abs}\) can absorb any set \(U\subseteq V(G)\setminus V(P_{abs})\) of size at most \(\mu n\), yielding a path \(P\) with \(V(P)=V(P_{abs})\cup U\) and the same end vertices as \(P_{abs}\).

\smallskip

\textbf{$\bullet$ Step 2: Reservoir construction.} 
We choose a small set $\mathcal R\subseteq V(G)\setminus V(P_{abs})$ such that every pair of non-adjacent vertices outside $\mathcal R$ can be connected by a path of length two whose internal vertex is an unused vertex of $\mathcal R$.

\textbf{$\bullet$ Step 3: Almost covering and final absorption.} 
We next apply the degree form of the Regularity Lemma to the underlying graph $U(G')$, where $G'=G\backslash(V(P_{abs})\cup \mathcal{R})$. This yields a family \( \mathcal{P} \) of vertex-disjoint paths in \( G' \) such that \( |\mathcal{P}| \) is bounded.
Using the reservoir set \( \mathcal{R} \), we can connect the dominant direction of all paths in \( \mathcal{P} \) together with \( P_{abs} \) into a cycle \( C \) in \( G \) that covers all but at most \( \mu n \) vertices, which can be absorbed by \( P_{abs} \) to yield a Hamilton cycle in \( G \).

\smallskip

\textbf{$\bullet$ Step 4: Discrepancy estimation.} 
Finally, we derive the lower bound on \( \sigma_{\max}(C) \) from \( \sigma_{\max}(\mathcal{P}) \).

\section{Useful lemmas and tools}\label{sec:tool}

\subsection{Regularity and blow-up tools}\label{subsec:regular}
 
We start with some definitions. For a graph $G$ and disjoint sets $A,B\subseteq V(G)$, let $E_G(A,B):=\{xy\in E(G)\mid x\in A,\ y\in B\}.$ When $G$ is clear, we omit the subscript.
The \textit{density} of a bipartite graph \(G[A,B]\) with vertex classes \(A\) and \(B\) is defined to be
\[
d_G(A,B) \coloneqq \frac{|E_G(A,B)|}{|A||B|}.
\]
 We often write \(d(A,B)\) if this is unambiguous. Given \(\varepsilon>0\), we say that the pair \((A,B)\) is \textit{\(\varepsilon\)-regular} if for all subsets \(X \subseteq A\) and \(Y \subseteq B\) with \(|X|>\varepsilon|A|\) and \(|Y|>\varepsilon|B|\) we have that \(|d(X,Y)-d(A,B)|<\varepsilon\). Given \(d \in [0,1]\) we say that the pair $(A,B)$ in \(G\) is \textit{\((\varepsilon,d)\)-super-regular} if it is \(\varepsilon\)-regular and furthermore $d_G(a,B) \ge d|B|$ for all \(a \in A\) and \(d_G(b,A) \ge d|A|\) for all \(b \in B\). Accordingly, a bipartite graph \(H\) with vertex classes \(A\) and \(B\) is called \((\varepsilon,d)\)-super-regular if the pair \((A,B)\) is \((\varepsilon,d)\)-super-regular in \(H\).

 In this paper, we adopt the degree form of the Regularity Lemma.
\begin{lemma}[Degree form of the Regularity Lemma]\label{lem:regularity}
Given any $\varepsilon\in(0,1)$ and $M'\in\mathbb{N}$, there exist $M=M(\varepsilon,M')\in\mathbb{N}$ and $n_0=n_0(\varepsilon,M')\in\mathbb{N}$ such that the following holds for all $n\ge n_0$. Let $G$ be an $n$-vertex graph and let $d\in[0,1]$. Then, there is a partition $\{V_0,V_1,\dots,V_k\}$ of $V(G)$ with $M'\le k\le M$ and a spanning subgraph $G'$ of $G$ such that

\medskip

\textnormal{($i$)} \(|V_0| \le \varepsilon n\),

\textnormal{($ii$)} \(|V_1| = \cdots = |V_k| =: m\),

\textnormal{($iii$)} \(\deg_{G'}(x) > \deg_G(x) - (d+\varepsilon)n\) for all vertices \(x \in G\),

\textnormal{($iv$)} for all \(1 \le i \le k\) the graph \(G'[V_i]\) is empty,

\textnormal{($v$)} for all \(1 \le i, j \le k\) and \(i \ne j\) the bipartite graph \(G'[V_i, V_j]\) whose vertex classes are \(V_i\) and \(V_j\) is \(\varepsilon\)-regular and has density either \(0\) or at least \(d\).
\end{lemma}

The vertex sets \(V_1, \ldots, V_k\) are called \emph{clusters}, \(V_0\) is called the \emph{exceptional set} and the vertices in \(V_0\) are called \emph{exceptional vertices}. 
%The last condition of the lemma says that all pairs of clusters are \(\varepsilon\)-regular in both directions (but possibly with different densities). 
We call the spanning graph \(G' \subseteq G\) given by \cref{lem:regularity} the \emph{pure graph}. Given clusters \(V_1, \ldots, V_k\) and a graph \(G'\), the \emph{reduced graph} \(R\) with parameters \((\varepsilon,d)\) has vertex set \(\{V_1,\ldots,V_k\}\), where \(V_iV_j\in E(R)\) precisely when \(G'[V_i,V_j]\) has density at least \(d\).

\begin{lemma}[\cite{tre}]\label{oreprop}
Let $G$ be a graph such that $d_G(x)+d_G(y)\geq c|G|$ for all $x\neq y\in V(G)$ such that $xy\notin E(G)$, where $c$ is a constant.
Let \(R\) be the reduced graph obtained by applying \cref{lem:regularity} to \(G\) with parameters \(\varepsilon\), \(d\) and \(M'\). Then $d_R(V_i)+d_R(V_j)>(c-2d-4\varepsilon)|R|$ for all $V_i\neq V_j\in V(R)$ such that $V_iV_j\notin E(R)$.
\end{lemma}

We will need two well-known properties about regular pairs.

\begin{proposition}\label{sli}
Let $0<\varepsilon<\xi\le1$. Let $G$ be a graph and let $(A,B)$ be $\varepsilon$-regular in $G$ with density $d$. If
$A'\subseteq A$ and $B'\subseteq B$ satisfy $|A'|\ge \xi|A|$ and $|B'|\ge \xi|B|$, then $(A',B')$ is $\varepsilon'$-regular in $G$ with density at least $d-\varepsilon$, where $\varepsilon'=\max\{2\varepsilon,\varepsilon/\xi\}.$
\end{proposition}

\begin{proposition}\label{p2}
Given $0<\varepsilon<d\le 1$, let $G$ be a graph and let $(A,B)$ be $\varepsilon$-regular in $G$ with density $d$. There are at most $\varepsilon|A|$ vertices $v\in A$ such that $d_G(v,B)\le (d-\varepsilon)|B|$, and at most $\varepsilon|B|$ vertices $w\in B$ such that $d_G(w,A)\le (d-\varepsilon)|A|$.
\end{proposition}

The following Blow-up Lemma was proposed by Koml\'{o}s, S\'{a}rk\"{o}zy and Szemer\'{e}di \cite{Komlos1}. We use it to find a family of vertex-disjoint paths that cover almost all vertices of $G$.

\begin{lemma}[\cite{Komlos1} Blow-up Lemma] \label{blowup}
For every \(d>0\) and \(\Delta,k\in\mathbb N\), there exists $\tau_0=\tau_0(d,\Delta,k)>0$
such that the following holds for every graph \(F\) on $k$ vertices, every collection of positive integers \(l_1,\ldots,l_k\), and every \(0<\tau\le \tau_0\). Let $F^\prime$ be the graph obtained from $F$ by replacing each vertex $i\in V(F)$ with a set $V_i$ of $l_i$ new vertices and connecting all vertices between $V_i$ and $V_j$ whenever $ij\in E(F)$. If $G_1$ is a spanning subgraph of $F^\prime$ where every pair $(V_i, V_j )_{G_1}$ is $(\tau, d)$-super-regular for every edge $ij\in E(F)$, then $G_1$ contains a copy of every subgraph $H$ of $F^\prime$ with $\Delta(H)\leq\Delta$. Moreover, this copy of $H$ in $G_1$ maps the vertices of $H$ to the same sets $V_i$ as the copy of $H$ in $F'$, i.e.\ if $h \in V(H)$ is mapped to $V_i$ by the copy of $H$ in $F'$, then it is also mapped to $V_i$ by the copy of $H$ in $G_1$.
 \end{lemma}

Alon and Shapira \cite{Alon} extended the Regularity Lemma to directed graphs, obtaining the Diregularity Lemma. We require the following refinement form of the Diregularity Lemma, see \cite[Lemma~3]{KKO}.

For a digraph \(D\) and two disjoint vertex sets \(A,B\subseteq V(D)\), let $E_D(A,B):=\{xy\in E(D):x\in A,\ y\in B\},$
and define the directed density from \(A\) to \(B\) by $d_D(A,B):=\frac{|E_D(A,B)|}{|A||B|}.$
Given \(\varepsilon>0\), we say that the pair \((A,B)\) is \(\varepsilon\)-regular in \(D\) if, for all \(X\subseteq A\) and \(Y\subseteq B\) with $|X|>\varepsilon|A|$ and $|Y|>\varepsilon|B|$ we have $|d_D(X,Y)-d_D(A,B)|<\varepsilon$ and $|d_D(Y,X)-d_D(B,A)|<\varepsilon.$

Let \(U_D(A, B)\) denote the bipartite graph with vertex classes \(A\) and \(B\) with $ab\in E\bigl(U_D(A, B)\bigr)$ 
if and only if $ab\in E(D),$
where \(a\in A\) and \(b\in B\).
%We call \(U_D(A\to B)\) the \(A\to B\) underlying bipartite graph of \(D\).
%In particular,
%$$
%d_{U_D(A\to B)}(A,B)=d_D(A,B).
%$$
Moreover, if \((A,B)\) is \(\varepsilon\)-regular in \(D\), then both
\(U_D(A, B)\) and \(U_D(B, A)\) are \(\varepsilon\)-regular bipartite graphs.

\begin{lemma}[\cite{KKO}]\label{lem:direg-refinement}
For every $\delta>0$ and every positive integer $T$, there exist
integers $Q=Q(\delta,T)$ and $n_0=n_0(\delta,T)$ such that the following
holds. Let $D$ be a digraph on $n\ge n_0$ vertices, and let
$\mathcal U=\{U_1,\ldots,U_t\}$ be a partition of $V(D)$ with $t\le T$.
Then $V(D)$ has a partition $\{W_0,W_1,\ldots,W_q\}$ such that
$1\le q\le Q$, $|W_0|\le\delta n$, $|W_1|=\cdots=|W_q|$, and every
$W_a$ with $a\in[q]$ is contained in some member of $\mathcal U$.
Moreover, all but at most $\delta q^2$ pairs $(W_a,W_b)$ with
$1\le a<b\le q$ are $\delta$-regular.
\end{lemma}

\subsection{Further auxiliary results}\label{subsec:HSthm}
 %\subsection{Hajnal-Szemer\'{e}di Theorem and Probabilistic tool} 

 We first need the Hajnal-Szemer\'{e}di theorem under the Ore-type condition.

\begin{lemma}[\cite{Kierstead}]\label{kk}
Let $k\in\mathbb{N}$ and $G$ be a graph. If $d_{G}(x)+d_{G}(y)\le2k+1$ for every edge $xy\in E(G)$, then $G$ has an equitable $k+1$-coloring.
\end{lemma}

\begin{lemma}\label{lem:ore-clique-partition}
For every integer $T\ge 4$ and every $\eta>0$, there exists $n_0$ such that the following holds for every $n\ge n_0$. Let $G$ be an $n$-vertex graph satisfying $\sigma_2(G)\ge(1+\eta)n.$
Then $V(G)$ admits a partition $\mathcal Q$ into cliques such that $2\le |Q|\le T$ for every $Q\in\mathcal Q,$ and $\sum_{Q\in\mathcal Q}\bigl(|Q|-1\bigr) \;\ge\; \sigma_2(G)/2-1-n/T.$
\end{lemma}

\begin{proof}

Let $s=\sigma_2(G)$ and $0<1/n_0\ll\eta$. When $G$ is complete, the conclusion is immediate. Assume therefore that $G$ is not complete and $q:=\lceil n-s/2\rceil$. 

%We first construct a partition of $V(G)$ into cliques, each of order at least $2$, while controlling the number of parts $q$.  
For every edge $xy\in E(\overline G)$, we have $d_{\overline G}(x)+d_{\overline G}(y)
 \;=\; 2(n-1)-\bigl(d_G(x)+d_G(y)\bigr)
 \;\le\; 2(n-1)-s
 \;\le\; 2(q-1)+1 .$
Since $s< 2(n-1)$, we have $q\ge 2$. 
Therefore, applying \cref{kk} to $\overline G$ with \(k=q-1\) yields an equitable \(q\)-coloring of \(\overline G\).
Each color class is independent in $\overline G$, hence a clique in $G$. 
Let $\mathcal Q'$ denote the corresponding partition
of $V(G)$ into $q$ cliques. Since the coloring is equitable, every clique in $\mathcal Q'$ has
order either $\lfloor n/q\rfloor$ or $\lceil n/q\rceil$.
Moreover, using $s\ge(1+\eta)n$, we obtain $q\le n-s/2+1\le (1-\eta)n/2+1< n/2.$
Hence every clique in $\mathcal Q'$ has order at least $2$.

Now refine $\mathcal Q'$ as follows. Keep every clique of order at most $T$ unchanged. If $Q\in\mathcal Q'$ has order $a>T$, split $Q$ into
$\lceil a/T\rceil$ parts whose sizes differ by at most one.
Since $T\ge4$, each new part has size between $2$ and
$T$. Let $\mathcal Q$ be the resulting partition. Then $|\mathcal Q|\le q+n/T ,$
and hence $\sum_{Q\in\mathcal Q}\bigl(|Q|-1\bigr)
= n-|\mathcal Q|
\ge n-q- n/T
\ge \sigma_2(G)/2-1- n/T.$
\end{proof}
%By now, we have already get a partition $\mathcal{Q}'$ of $V(G)$, and each $Q\in \mathcal{Q}'$ is a clique of order at least $2$. For a fix $T\ge4$, there may be some cliques more than $T$ vertices, and split its vertex set into $c:=\lceil a/T\rceil$parts whose sizes differ by at most one.  Since $T\ge 4$ and $a>T$,  Moreover $c-1 < a/T$. If the original partition has $q$ cliques, the refined partition $\mathcal{Q}$ therefore satisfies $|\mathcal{Q}| \le q + n/T.$

The next lemma is a classical result on tournaments.
\begin{lemma}[\cite{re}]\label{re}
Every tournament contains a directed Hamilton path.
\end{lemma}

% \subsection{Probabilistic tools}
The following result will also be useful in proving the absorbing lemma. It packages the selection of absorbers into a convenient form and thereby simplifies the argument.
 
\begin{lemma}[\cite{Chang}]\label{disjoint-tuples}
Let \(\sigma\) be a real number with \(0<\sigma<1\) and \(t\) be an integer with \(1\leq t\leq 4\). Then there exists an integer \(n_0\) such that whenever \(n\geq n_0\) the following holds. Let \(G\) be an oriented graph on \(n\) vertices and \(X\) be a vertex subset with \(|X|\geq n/2\). Let \(S\subseteq V(G)\times V(G)\) be a set of pairs of vertices (not necessarily distinct). Suppose that \(\mathcal{A}(u,v)\) is a family of  ordered $t$-tuples of distinct vertices of $X$  such that \(|\mathcal{A}(u,v)|\geq \sigma n^t\) for every \((u,v)\in S\). Then there exists a family \(\mathcal{F}\subseteq\bigcup \mathcal{A}(u,v)\) of pairwise vertex-disjoint ordered \(t\)-tuples satisfying $|\mathcal{F}|\leq 2^{-6}\sigma n$ and $|\mathcal{A}(u,v)\cap\mathcal{F}|\geq 2^{-10}\sigma^{2}n$
for all \((u,v)\in S\).
\end{lemma}

 \section{Proof of \cref{th2}}
Let $0<\eta\le 2\gamma\ll 1$. We first consider the case when $G$ is a tournament.
By \cref{re}, \(G\) contains a oriented Hamilton cycle \(C\) with \(\sigma_{\max}(C)= n-1=\sigma_2(G)/2\).
Hence the conclusion follows. Thus, we may assume that $G$ is not a tournament.

Suppose that $n\le\sigma_2(G)\le (1+\eta) n$.
  Since \(\sigma_2(U(G))=\sigma_2(G)\ge n\), Ore's theorem applied to \(U(G)\) yields an oriented Hamilton cycle \(C\) in $G$, and trivially $\sigma_{\max}(C)\geq n/2$. It follows readily from $\eta\le 2\gamma$ that $\sigma_{\max}(C)\ge n/2\geq (1+\eta)n/2-\gamma n\ge\sigma_{2}(G)/2-\gamma n.$
  Thus, we may assume \( \sigma_2(G) \geq (1+\eta)n \) for the rest of the proof.
 
\subsection{Preparation of \cref{th2}}\label{sec:ore}

In this section, we establish three key lemmas used for 
the proof of \cref{th2}: an almost covering lemma, an absorbing lemma and a reservoir lemma.

\smallskip
\noindent\textbf{4.1.1 Almost Covering Lemma.} 

In this part we prove the following almost covering lemma. 

\begin{lemma}\label{lem:directional-lifting}
For every integer $T\ge2$ and every $d,\xi>0$, there exist $\varepsilon>0$ and $L^*,m_0\in\mathbb N$ such that the following holds. Let $2\le t\le T$, and let $G$ be a balanced $t$-partite
oriented graph with vertex classes $V_1,\ldots,V_t$, each of size $m\ge m_0$. Suppose that, for all distinct $i,j\in[t]$, the bipartite graph $U(G)(V_i,V_j)$ is $\varepsilon$-regular and has density at least
$d$. 
Then $G$ contains a family $\mathcal C$ of at most
$L^*$ vertex-disjoint cycles covering all but at most
$\xi m$ vertices of each $V_i$ and satisfying
$\sigma_{\max}(\mathcal C)\ge(t-1)(1-\xi)m$.
\end{lemma}

\begin{proof}
Let 
\begin{equation}\label{3}
    0<1/m_0\ll\varepsilon\ll1/L^*\ll\rho\ll d,\xi,1/T\le1.
\end{equation}

\begin{claim}\label{cl:multipartite-refinement}
There exist an integer $L$ with $1\le L\le L^*$ and, for every
$i\in[t]$, a partition $V_i=V_i^0\cup V_i^1\cup\cdots\cup V_i^L$ such that $|V_i^0|\le\rho m$, while the sets
$V_i^1,\ldots,V_i^L$, which we call the \emph{cells} of $V_i$, all have the same size $m'\ge 2\varepsilon m$. Moreover, for every $i\ne j\in[t]$, all but
at most $\rho L^2$ pairs $(V_i^a,V_j^b)$ with $a,b\in[L]$ are
$\rho$-regular.
\end{claim}

\begin{proof}
Apply \cref{lem:direg-refinement} with
$\delta=\rho^3$ %$Q(\rho^3,T)\le L^*$ 
and the partition
$\{V_1,\ldots,V_t\}$ of $V(G)$. We obtain a partition
$\{W_0,W_1,\ldots,W_q\}$ of $V(G)$, where $q\le L^*$ such that the exceptional set
 $|W_0|\le\rho^3tm$, and $|W_1|=\cdots=|W_q|=:m'$. 
 
 For each $i\in[t]$, let $l_i$ be the number of non-exceptional sets contained in $V_i$, 
 and $L:=\min_{i\in[t]}l_i$. Since $|W_0|<m$, we have $1\le L\le q\le L^*$. For each $i\in[t]$, retain any $L$ clusters contained in $V_i$, denote them by $V_i^1,\ldots,V_i^L$, and put all remaining vertices of $V_i$ into $V_i^0$. Consequently,  $|V_i^0|\le m-Lm'\le\max_{i\in[t]}|W_0\cap V_i|\le|W_0|\le\rho^3tm\le\rho m$ and $ m'\ge(1-\rho) m/L\ge 2\varepsilon m.$
Since $l_i\le m/m'\le L/(1-\rho)\le2L$ for every $i\in[t]$, we have $q=\sum_{i=1}^t l_i\le2tL\le2TL$. Together with \eqref{3}, this implies that the number of pairs that are not $\rho$-regular is at most $\rho^3q^2\le4\rho^3T^2L^2\le\rho L^2$. This proves the claim.
\end{proof}

 Call a pair
$(V_i^a,V_j^b)$ \emph{good} if it is $\rho$-regular in $G$. For a good pair $(V_i^a,V_j^b)$, since
$m'\ge2\varepsilon m$, the $\varepsilon$-regularity of the pair $(V_i,V_j)$ implies that $d_{U(G)}(V_i^a,V_j^b)=d_{G}(V_i^a,V_j^b)+d_{G}(V_j^b,V_i^a)\ge d-\varepsilon$.
Hence one of the two directed densities is at least $(d-\varepsilon)/2\ge d/3$. Fix one such direction for every good pair and call it its \emph{good direction}.

Construct a $t$-partite auxiliary graph $H$ whose $i$-th partite set is $\mathcal {V}_i = \{V_i^1,\dots,V_i^L\}$ for each $i\in[t]$, where two vertices $V_i^a,V_j^b$ are adjacent in  $H$ precisely when $(V_i^a,V_j^b)$ forms a good pair in $G$.

\begin{claim}\label{cl1}
The graph $H$ contains $h$ vertex-disjoint copies of $K_t$ with $h\ge L-2(t-1)\sqrt\rho\,L.$
\end{claim}
\begin{proof}
For fixed $i\neq j\in[t]$, at most $\sqrt\rho L$ vertices in $\mathcal V_i$ have more than $\sqrt\rho L$ non-neighbors in $\mathcal V_j$. Delete from each $\mathcal V_i$ every vertex with more than $\sqrt\rho L$ non-neighbors in some other part, and trim the remaining parts to a common size $L'\ge L-(t-1)\sqrt\rho L$. Every remaining vertex has at most $\sqrt\rho L$ non-neighbors in each other partite set.

Greedily choose copies of $K_t$, one vertex from each partite set.
If the unused vertex set of every part has common size $s>(t-1)\sqrt\rho L$, then after choosing cells from the first $r-1$ parts there remain at least $s-(r-1)\sqrt\rho L>0$ vertices in the $r$-th part adjacent to all previously chosen vertices. The process therefore stops only when at most $(t-1)\sqrt\rho L$ vertices remain in each part, giving $h$ pairwise vertex-disjoint $K_t$ in $H$, where $h\ge L-2(t-1)\sqrt\rho L$.
\end{proof}

Fix one such copy of $K_t$, with vertex set $\{A_1',\ldots,A_t'\}$, and orient each of its edges according to the good direction of the corresponding pair. This gives a tournament on $t$ vertices.
By \cref{re}, after relabeling, we may assume that
$A_1'\dots A_t'$ is a directed Hamilton path in $T_t$. If $t\ge3$, let $C_1'=A_1'A_2'\cdots A_t'A_1'$ be an oriented cycle. If $t=2$, let $C_1'=A_1'A_2'$ be an edge. In either case, the good direction between $A_i'$ and $A_{i+1}'$ is $A_i'\to A_{i+1}'$ for every $i\in[t-1]$. Thus
\begin{equation}\label{ci}
    \sigma_{\max}(C_1')\ge t-1.
\end{equation}

\begin{claim}\label{cl3}
    There exist \(A_i\subseteq A'_i\), \(i\in[t]\), of a common size $\ell\ge (1-2\rho)m',$
such that, for every edge \(A'_iA'_j\in E(C'_1)\), if \(A'_i\to A'_j\) is the good direction of the pair \((A'_i,A'_j)\), then the bipartite graph $U_G(A_i,A_j)$
is \((2\rho,d/5)\)-super-regular.
\end{claim}
\begin{proof}
For every \(A_i'A_j'\in E(C_1')\), suppose that the good direction of the pair is \(A'_i\to A'_j\). Thus, \(U_G(A'_i, A'_j)\) is \(\rho\)-regular and has density at least \(d/3\).

Delete from $A_i'$ every vertex having fewer than $(d/3-\rho)m'$ out-neighbors in $A_j'$, and delete from $A_j'$ every vertex having fewer than $(d/3-\rho)m'$
in-neighbors from \(A'_i\).  Since every cell lies in at most two of these pairs, and by \cref{p2}, at most $2\rho m'$ vertices are deleted from any cell.  
Let \(\ell\) be the minimum size of the remaining cells in this copy of \(K_t\), and trim all these cells to the common size \(\ell\). Denote the resulting trimmed cells by \(A_1,\ldots,A_t\).
Hence, $\ell\ge(1-2\rho)m'.$

By \cref{sli}, applied with $\xi=1-2\rho,\varepsilon=\rho,d=d/3$, each \(U_G(A_i, A_j)\) is $2\rho$-regular and has
density at least $d/3-\rho\ge d/4$. 
Moreover, in \(U_G(A_i,A_j)\), by \eqref{3} and $\ell\ge(1-2\rho)m'$, every vertex surviving in either part has at least $(d/3-\rho)m'-2\rho m' =(d/3-3\rho)m'\ge(d/4-6\rho)\ell\ge d\ell/5$ neighbors in the opposite part.
Therefore \(U_G(A_i, A_j)\) is \((2\rho,d/5)\)-super-regular. 
\end{proof}

Observe that all the new cells
$A_i=\{a_i^1,\ldots,a_i^\ell\}$, $i\in[t]$, have the same size
$\ell$. Consider the $\ell$-blow-up of  $C_1'$. It contains the Hamilton cycle
\[
C_1:=
a_1^1a_2^1\cdots a_t^1
a_1^2a_2^2\cdots a_t^2
\cdots
a_1^\ell a_2^\ell\cdots a_t^\ell a_1^1.
\]
An application of \cref{blowup} embeds $U(C_1)$ into $U(G)$ through the super-regular pairs obtained above.
By their definition, every edge of \(C_1\) is oriented in \(G\) according to the good direction of the corresponding pair of cells.
More precisely, for every $p\in[\ell]$ and every $i\in[t-1]$, the
edge $a_i^pa_{i+1}^p$ is directed from $a_i^p$ to $a_{i+1}^p$,
since the good direction between $A_i$ and $A_{i+1}$ is
$A_i\to A_{i+1}$. On the other hand, the edges
$a_t^pa_1^{p+1}$ for $p\in[\ell-1]$, together with
$a_t^\ell a_1^1$, are all oriented according to the good direction
between $A_t$ and $A_1$, which may or may not agree with the cyclic
order of $C_1$. Thus, together with \eqref{ci}, $\sigma_{\max}(C_1)\ge \ell(t-1).$

Applying the same argument to each of the remaining copies of $K_t$
yields a family of vertex-disjoint cycles $\mathcal{C}=\{C_1,C_2,\dots,C_{h}\}$ where  $h\le L \le L^*$. From now on, redefine $\ell$ to be the minimum of the values of $\ell$
obtained from these $h$ copies of $K_t$.
 Hence
\begin{equation}\label{sigmaxc}
    \sigma_{\max}(\mathcal{C})\ge\sum_{i=1}^h \sigma_{\max}(C_i)\ge\ell h(t-1).
\end{equation}
Furthermore, by \eqref{3} and \cref{cl1},
and recalling that $Lm' \ge (1 - \rho)m$ and
$\ell \ge (1 - 2\rho)m'$, the number of covered vertices in each $V_i$ is at least $\ell h \ge \bigl(1 - 2T\sqrt{\rho}\bigr)(1 - 2\rho)(1 - \rho)m \ge (1 - \xi)m.$
Together with \eqref{sigmaxc}, we have $\sigma_{\max}(\mathcal{C}) \ge (1 - \xi)(t-1)m.$
\end{proof}

\begin{lemma}[Almost Covering Lemma]\label{lem:almost-covering}
Suppose \(1/n\ll 1/K \ll\xi \ll\eta \ll 1\). Let \(G\) be an $n$-vertex oriented graph satisfying $\sigma_{2}(G)\ge(1+\eta )n$. 
Then there exists a family \(\mathcal P\) of at most \(K\) vertex-disjoint paths in \(G\) that covers all but at most \(\xi n\) vertices and satisfies $\sigma_{\max}(\mathcal{P}) \ge \frac{\sigma_2(G)}{2} - \xi n$.
\end{lemma}

\begin{proof}

Let
\begin{equation*}
0<\frac1n\ll\frac1K\ll\frac1M\ll\frac1{M'}
    \ll\varepsilon\ll d,\frac1T\ll\xi\ll\eta\ll1.
\end{equation*}
Apply Lemma~\ref{lem:directional-lifting} with parameters
\(T,d,\xi/2\), and let \(\varepsilon_0,L^*,m_0\) be the resulting
constants.
Choose $\varepsilon\le\varepsilon_0$, and let
$M=M(\varepsilon,M')$ be given by
\cref{lem:regularity}.
Take $K$ sufficiently large that $L^*M\le K$,
and $n$ such that
$K\le\xi n/4$ and $(1-\varepsilon)n/M\ge m_0$.

If \(G\) is a tournament, then by \cref{re}, \(G\) contains a directed Hamilton path. Taking this path as \(\mathcal P\), we have $\sigma_{\max}(\mathcal P)=n-1\ge \sigma_2(G)/2-\xi n,$
and the conclusion follows. Hence we may assume that \(G\) is not a tournament.

 Apply \cref{lem:regularity} to $U := U(G)$ with parameters
$(\varepsilon, d, M')$  to obtain a partition \(V_0, V_1, \ldots, V_k\) of \(V(G)\) with \(M' \le k\le M\). Let $U'$
be the resulting pure graph and let $G'$ be the oriented spanning
subgraph of $G$ obtained by giving every edge of $U'$ its original
orientation in $G$. Let \(R\) be the reduced graph with parameters \((\varepsilon, d)\). By Lemma \ref{oreprop} 
and the choice $\varepsilon\ll d\ll \eta$, we have\footnote{If $R$ is a complete graph, it still satisfies \eqref{sigma2R}.}
 \begin{equation}\label{sigma2R}
     \sigma_2(R)\ge( \sigma_2(G)/n - 2d-4\varepsilon ) k\ge(1+\eta/2)k.
 \end{equation}
Apply \cref{lem:ore-clique-partition} to $R$ to obtain a partition $\mathcal{Q}$ of $V(R)$ into cliques of orders between $2$ and $T$ such that
\begin{equation}\label{eq:clique-budget}
\sum_{Q \in \mathcal{Q}} (|Q| - 1) \ge \frac{\sigma_2(R)}{2} - 1 - \frac{k}{T}.
\end{equation}

Fix $Q\in\mathcal Q$, say
$Q=\{V_1,\ldots,V_t\}$ with $2\le t\le T$.
Every pair corresponding to an edge of $R$ is $\varepsilon$-regular in $U'$ with density at least $d$.
Applying \cref{lem:directional-lifting} with parameter $\xi/2$, we obtain a
family $\mathcal C_Q$ of at most $L^*$
vertex-disjoint cycles leaving at most
$\xi m/2$ vertices uncovered in each
$V_i$, $i\in[t]$.
Moreover,
\begin{equation}\label{eq:sigma-clique}
\sigma_{\max}(\mathcal C_Q)
\ge(|Q|-1)(1-\xi/2)m.
\end{equation}

Let $\mathcal C=\bigcup_{Q\in\mathcal Q}\mathcal C_Q$
and write $\mathcal C=\{C_1,\ldots,C_s\}$, where
$s\le L^*|\mathcal Q|\le L^*k\le K$.
For each $i\in[s]$, delete one edge from $C_i$
to obtain a path $P_i$, and put
$\mathcal P=\{P_1,\ldots,P_s\}$.
Then 
$\sigma_{\max}(\mathcal P)
\ge\sigma_{\max}(\mathcal C)-s
\ge\sigma_{\max}(\mathcal C)-K$.
Since $|V_0|\le\varepsilon n$ and the cycles
leave at most $\xi m/2$ vertices uncovered
in each of the $k$ clusters, we have
$|V(G)\setminus V(\mathcal P)|
\le |V_0|+\xi km/2
\le(\varepsilon+\xi/2)n\le\xi n$.

We next bound the edges of $\mathcal {P}$ oriented in the dominant direction. By \eqref{sigma2R}, \eqref{eq:clique-budget}  and \eqref{eq:sigma-clique} and recalling that $\sigma_2(G) \le 2(n - 1)$, $km \ge (1 - \varepsilon)n$, we obtain
\begin{equation*}
\begin{aligned}
&\sigma_{\max}(\mathcal{C})
\ge\sum_{Q \in \mathcal{Q}}\sigma_{\max}(\mathcal{C}_Q)
\ge \Bigl(1-\frac{\xi}{2}\Bigr)m\Bigl(\frac{\sigma_2(R)}{2}-1-\frac{k}{T}\Bigr)\\
&\ge \Bigl(1-\frac{\xi}{2}\Bigr)km\Bigl(\frac{\sigma_2(G)}{2n}-d-2\varepsilon-\frac{1}{k}-\frac{1}{T}\Bigr)
\ge \frac{\sigma_2(G)}{2}-\frac{3\xi n}{4}.
\end{aligned}
\end{equation*}
Together with
$\sigma_{\max}(\mathcal P)
\ge\sigma_{\max}(\mathcal C)-K$ and $K\le\xi n/4$, this gives $\sigma_{\max}(\mathcal P)
\ge\sigma_2(G)/2-\xi n$, as required.
\end{proof}

\smallskip
\noindent\textbf{4.1.2 Absorbing and Connecting Lemmas.} 

In this part we establish the absorption lemma and the connecting lemma. 

\begin{lemma}[Connecting Lemma]\label{connecting}
Suppose \(0 < 1/n \ll\eta \ll 1\). Let \(G\) be an \(n\)-vertex oriented graph satisfying $\sigma_{2}(G)\ge(1+\eta )n$.  If $u,v\in V(G)$ are non-adjacent, then \(G\) contains at least \(\eta n\) paths of length \(2\) with end vertices  $u$ and $v$.
\end{lemma}

 \begin{proof}
 This follows from $|N(u)\cap N(v)|\ge(1+\eta)n-n\ge \eta n$, since \(\sigma_2(G) \ge (1+\eta)n\).
 \end{proof}

\begin{lemma}[Absorbing Lemma]\label{absorbing}
Suppose \(0 < 1/n \ll \mu\ll\eta \ll 1\). 
Let \(G\) be an \(n\)-vertex oriented graph satisfying $\sigma_{2}(G)\ge(1+\eta )n$.   Then there exists a  path \(P_{abs}\) with at most \( 8\sqrt{\mu} n\) vertices such that, for every \(W \subseteq V(G) \setminus V(P_{abs})\) with \(|W| \le  \mu n\), \(G[V(P_{abs}) \cup W]\) contains a spanning path having the same end vertices as \(P_{abs}\).
\end{lemma}

To prove the Absorbing Lemma, we need the following definitions and the two-step absorption method, which is adapted from \cite{Chang}.

\begin{definition}
Let $0<\alpha_1\ll 1$ and \(G\) be an \(n\)-vertex oriented graph. A \emph{strong absorber} of a pair of vertices \((u,v)\) (not necessarily distinct)  is an ordered pair \((a,b)\)  of distinct vertices such that \(ab, au, bv \in E(U(G))\). We say that \((u,v)\) is \emph{\(\alpha_1\)-strongly absorbable} if it has at least \(\alpha_1 n^2\) strong absorbers. When \(u=v\), we identify the pair \((u,u)\) with the single vertex \(u\) and use the same terminology.
\end{definition}

\begin{definition}
Let $0<\alpha_2\ll\alpha_1\ll 1$ and \(G\) be an \(n\)-vertex oriented graph. An \emph{\(\alpha_1\)-weak absorber} of a pair of vertices \((u,v)\) (not necessarily distinct) is an ordered $4$-tuple \(a,a',b',b \in V(G)\) such that \(aa', au, vb, b'b \in E(U(G))\) and the ordered pair \((a',b')\) is \emph{\(\alpha_1\)-strongly absorbable}. We say that \((u,v)\) is \((\alpha_2,\alpha_1)\)-weakly absorbable if it has at least \(\alpha_2 n^4\)  \(\alpha_1\)-weak absorbers. When \(u=v\), we identify the pair \((u,u)\) with the single vertex \(u\) and use the same terminology.
\end{definition}

The next lemma shows that every vertex is either strongly absorbable or weakly absorbable.

\begin{lemma}\label{absorbability}
Let \(1/n \ll \alpha_2 \ll \alpha_1 \ll \eta\ll1\). Let \(G\) be an \(n\)-vertex oriented graph with $\sigma_2(G)\ge (1+\eta)n$. Then every vertex \(v \in V(G)\) is either \(\alpha_1\)-strongly absorbable or \((\alpha_2,\alpha_1)\)-weakly absorbable.
\end{lemma}

We first deduce \cref{absorbing} from \cref{absorbability}, and then prove the latter.

\begin{proof}[\textbf{Proof of Lemma \ref{absorbing}:}]
Choose \(\alpha_1\) and \(\alpha_2\) such that $1/n \ll \alpha_2 \ll \alpha_1 \ll \eta\ll1$ and $\alpha_2 := 32\sqrt{\mu}$.

We identify each vertex $u$ with the pair $(u,u)$, allowing single-vertex absorption to be treated as a special case of pair absorption. 
Let $\mathcal{U}_s$ and $\mathcal{U}_w$ denote the sets of pairs $(u,v)\in V(G)\times V(G)$ that are $\alpha_1$-strongly absorbable and $(\alpha_2,\alpha_1)$-weakly absorbable, respectively.
 For each pair \((u,v)\in V(G)\times V(G)\), write \(\mathcal{A}_s(u,v)\) and \(\mathcal{A}_w(u,v)\) for the sets of strong absorbers and weak absorbers of \((u,v)\), respectively.

Apply \cref{disjoint-tuples} to \(\mathcal{U}_w\)  with $X=V(G)$, \(\sigma = \alpha_2\) and \(t = 4\). 
This yields a family \(\mathcal{F}_w\) of weak absorbers such that
\begin{equation}\label{Fw}
   |\mathcal{F}_w| \leq 2^{-6} \alpha_2 n, \quad |\mathcal{A}_w(u,u) \cap \mathcal{F}_w| \geq 2^{-10} \alpha_2^2 n =\mu n\quad\text{ for all }\quad (u,u) \in \mathcal{U}_w. 
\end{equation}
For each $(u,v)\in \mathcal{U}_s$, let
$\mathcal A'_s(u,v)\subseteq\mathcal A_s(u,v)$ be obtained by
removing every ordered pair $(a,b)$ whose vertex set
$\{a,b\}$ intersects $V(\mathcal F_w)$. The number of pairs removed in this way is at most $|V(\mathcal{F}_w)|\cdot 2n\le 2^{-3}\alpha_2 n^2$.
Since $\alpha_2\ll\alpha_1$, we have $|\mathcal{A}_s^{'}(u,v)|\ge \alpha_1 n^2-2^{-3}\alpha_2 n^2\ge \alpha_2 n^2.$
We may now choose a  family \(\mathcal{F}_s\) of strong absorbers by \cref{disjoint-tuples} with $X=V(G)\backslash V(\mathcal{F}_w)$, \(\sigma = \alpha_2\) and \(t = 2\) with respect to \(\mathcal{U}_s\) and \(\mathcal{A}'_s(u,v)\) such that 

\begin{equation}\label{Fs}
   |\mathcal{F}_s| \leq 2^{-6} \alpha_2 n, \quad |\mathcal{A}_s(u,v) \cap \mathcal{F}_s| \geq 2^{-10} \alpha_2^2 n=\mu n \quad\text{ for all }\quad (u,v) \in \mathcal{U}_s. 
\end{equation}

Now we show that $|V(P_{abs})|\le8\sqrt{\mu} n$.
The absorbing path \(P_{abs}\) is built by taking all strong absorbers (each edge in \(\mathcal{F}_s\)) and, for every weak absorber \((a, a', b', b)\) in \(\mathcal{F}_w\), the two edges \(aa'\) and \(b'b\) together with an \((a', b')\)-path which is disjoint from any other absorber in $\mathcal{F}_s\cup\mathcal{F}_w.$ 
These \((a',b')\)-paths can be chosen to be pairwise internally vertex-disjoint, with all their internal vertices lying outside \(V(\mathcal F_s)\cup V(\mathcal F_w)\). Indeed, whenever \(a',b'\) are non-adjacent, by \cref{connecting}, they have at least \(\eta n\) common neighbors. On the other hand, the total number of vertices in \(V(\mathcal F_s)\cup V(\mathcal F_w)\), together with all previously used internal vertices, is at most
$6\bigl(|\mathcal F_s|+|\mathcal F_w|\bigr)
\le \alpha_2 n/4
\ll \eta n.$
Hence, we may choose a common neighbor of \(a'\) and \(b'\) outside all previously used vertices and use it to form an \((a',b')\)-path of length two. 
 For the same reason, we can connect all the strong and weak absorbers into a single path avoiding all previously used vertices. Consequently, by \eqref{Fw} and \eqref{Fs}, we have
\begin{equation*}
    |V(P_{abs})|\le2|\mathcal{F}_s|+4|
\mathcal{F}_w|+|\mathcal{F}_w|+(|\mathcal{F}_s|+|\mathcal{F}_w|)\le6|\mathcal{F}_s \cup \mathcal{F}_w| \leq \alpha_2 n/4=8\sqrt{\mu}n.
\end{equation*}

It remains to verify the absorbing property of \(P_{abs}\).
By Lemma \ref{absorbability}, 
each vertex is strongly or weakly absorbable. Hence, any vertex in \(W\) can be absorbed into \(P_{abs}\) with one strong absorber and, if necessary, one weak absorber.
 We choose these absorbers greedily for the vertices of \(W\). Before considering any vertex \(w\in W\), fewer than \(|W|\le \mu n\) absorbers from each of \(\mathcal F_s\) and \(\mathcal F_w\) have been used. Hence, by \eqref{Fw} and \eqref{Fs}, if \(w\) is strongly absorbable, there is an unused strong absorber of \(w\). If \(w\) is weakly absorbable, there is an unused weak absorber of \(w\), and there is also an unused strong absorber. Thus all vertices of \(W\) can be absorbed. More precisely, the absorption of a single vertex \(w\in W\) proceeds as follows.

\medskip
\noindent\textbf{Strong absorption:} If $w$ is $\alpha_1$-strongly absorbable, take a strong absorber $(a,b) \in \mathcal{F}_s$ of $w$, and insert $w$ into the edge $ab$. 

\medskip

\noindent\textbf{Weak absorption:} If $w$ is $(\alpha_2,\alpha_1)$-weakly absorbable, take a weak absorber $(a,a',b',b) \in \mathcal{F}_w$ of $w$.
Then $aa',\, aw,\, wb,\, b'b \in E(U(G))$ and $(a',b')$ is $\alpha_1$-strongly absorbable. By construction, $P_{abs}$ contains the segment $a \; Q \; b,$ where $Q$ is a $(a',b')$-path of length 1 or 2 used in constructing $P_{abs}$.
We absorb $w$ by replacing the segment $a\,  Q\, b$ by the path $awb$. Since $a',b'$ is $\alpha_1$-strongly absorbable, take a strong absorber $(c,d) \in \mathcal{F}_s$ of $(a',b')$, and replace the edge $cd$ of the path by $c\,  Q\, d$, which inserts the segment $ Q$ into $cd$.
\end{proof}

\begin{proof}[\textbf{Proof of \cref{absorbability}:}]
Call a vertex \emph{good} if
$d(v)\ge(1+\eta)n/2$ and \emph{bad} otherwise. Fix
$v\in V(G)$, let $A:=N(v)$, and let $B\subseteq A$ be the set of
bad vertices in $A$. Since
$\delta(G)\ge\sigma_2(G)-(n-1)>\eta n$, we have
$|A|>\eta n$.
   
    \emph{Case 1: $|B|\ge 2\sqrt{\alpha_1}n$.} 
    The underlying graph $U(G)[B]$ is complete, since two non-adjacent vertices $x,y\in B$ would imply $d(x)+d(y)<(1+\eta)n$. 
    Every $2$-subset of $B$ gives at least one ordered strong absorber of $v$. Therefore $v$ has at least $\binom{|B|}{2}\ge\alpha_1n^2$ strong absorbers and $v$ is $\alpha_1$-strongly absorbable.

\emph{Case 2: $|B|< 2\sqrt{\alpha_1}n$.}
If $v$ is good, then $A$ contains more than $(1+\eta)n/2-2\sqrt{\alpha_1}n>n/2$ good vertices. For every good $x\in A$, the vertices $x$ and $v$ have at least $d(x)+d(v)-n\ge\eta n$ common neighbors $z$, and since $xz,zv,xv\in E(U(G))$, each ordered pair $(x,z)$ is a strong absorber of $v$. Hence $v$ has more than $1/2\cdot n/2\cdot\eta n>\alpha_1 n^2$ strong absorbers.

We may therefore assume that $v$ is bad. Every other bad vertex lies in $A$. Hence $A$ contains at least $\eta n-2\sqrt{\alpha_1} n\ge \eta n/2$ good vertices. For every two  
good vertices $x\ne y\in A$, let $Z$ be the set of
good vertices in $N(x)\cap N(y)$. Consequently, 
\begin{equation}\label{c}
    |Z|\ge|N(x)\cap N(y)|-2\sqrt{\alpha_1}n-1\ge\eta n/2.
\end{equation}

We claim that every two distinct vertices $x',y'\in Z$ is $\alpha_1$-strongly absorbable. Since all bad vertices except $v$ lie in $A$,  $N(y')$ contains at least $|N(y')|-2\sqrt{\alpha_1}n-1\ge(1+\eta)n/2-3\sqrt{\alpha_1}n>n/2$ good vertices $z$.
For every such $z$, the good vertices
$x'$ and $z$ have at least $d(x')+d(z)-n\ge\eta n$ common neighbors $p$, and as $pz,px',zy'\in E(U(G))$,
$\{p,z\}$ is a strong absorber for $(x',y')$ and $(x',y')$ has more than $1/2\cdot n/2\cdot\eta n\ge\alpha_1 n^2$ strong absorbers.

For good vertices $x\ne y\in A$ and $x'\ne y'\in Z$, the ordered $4$-tuple
$(x,x',y',y)$ is an $\alpha_1$-weak absorber for $v$.
Indeed, $xx',xv,vy,y'y\in E(U(G))$, and $(x',y')$ is
$\alpha_1$-strongly absorbable. 
By \eqref{c} and the fact that \(A\) contains at least \(\eta n/2\) good vertices, there are at least $1/24\cdot (\eta n/3)^4\ge\alpha_2n^4$ such weak absorbers. Hence $v$ is $(\alpha_2,\alpha_1)$-weakly absorbable.
\end{proof}

\smallskip
\noindent \textbf{4.1.3 Reservoir Lemma.} 

The following lemma shows that any two vertices can be joined by a path with at most one internal vertex while preserving the existing structures in $G$.

\begin{lemma}[Reservoir Lemma]\label{reservoir}
Suppose \(0 < 1/n \ll \tau \ll \mu \ll\eta\ll 1\). Let \(G\) be an \(n\)-vertex oriented graph satisfying $\sigma_{2}(G)\ge(1+\eta )n$. For any vertex subset \(U \subseteq V(G)\) with \(|U| \le 8\sqrt{\mu}n\) there exists a set of vertices \(\mathcal{R} \subseteq V(G) \setminus U\) of size at most \(2^{-6} \tau n\) having the following property:
For every \(S \subseteq \mathcal{R}\) with \(|S| \le 2^{-11} \tau^2 n\) and for every pair of non-adjacent vertices \(u\) and $v$ in \(V(G) \setminus \mathcal{R}\), there exists an \((u,v)\)-path with one internal vertex from \(\mathcal{R} \setminus S\).
\end{lemma}

\begin{proof}
Let $U \subseteq V(G)$ with $|U| \leq 8\sqrt{\mu} n$. Denote by $\mathcal{Y}\subseteq V(G)\times V(G)$ the set of all non-adjacent pairs $(u,v)$ in $G$. 
For each $(u,v)\in \mathcal{Y}$, let $\mathcal{C}(u, v)$ be the set of vertices in \(V(G)\setminus U\) that form a path of length 2 with end vertices  \(u\) and \(v\).
By \cref{connecting}, we have $|\mathcal{C}(u, v)| \geq \eta n - |U| \geq \tau n$.

We now apply \cref{disjoint-tuples} where \(V(G) \setminus U\) plays the role of \(X\) and \(\sigma = \tau\), \(t = 1\). 
This yields a family of singletons  $\mathcal{F}\subseteq \bigcup_{(u,v)\in \mathcal{Y}} \mathcal{C}(u,v)$ such that
\begin{equation}\label{F}
    |\mathcal{F}| \leq 2^{-6} \tau n \quad \text{and} \quad |\mathcal{C}(u, v) \cap \mathcal{F}| \geq 2^{-10} \tau^2 n \quad \text{for all } (u, v) \in \mathcal{Y}.
\end{equation}
Identifying each singleton \(\{w\}\) with its unique element \(w\), set \( \mathcal{R} := V(\mathcal{F}) \). Now consider any subset \( S \subseteq \mathcal{R} \) with \( |S| \leq 2^{-11} \tau^2 n \).  By \eqref{F}, $|\mathcal{C}(u, v) \cap \mathcal{F}| \geq 2^{-10} \tau^2 n> |S|$, for any non-adjacent \( (u, v) \in \mathcal{Y} \).
%guarantees the existence of at least \( 2^{-10} \tau^2 n \) vertices adjacent to both \( u \) and \( v \), a quantity exceeding \( |S| \). 
Consequently, one can select a vertex in \( \mathcal{R} \setminus S \) that forms a $(u, v)$-path of length two. This completes the proof.
\end{proof}

\subsection{Completion of \cref{th2}}\label{sec:proof1}

 %Let \(K \in \mathbb{N}\) and introduce additional constants such that
 Let
\begin{equation}\label{1}
    1/n \ll 1/K \ll \xi \ll \tau \ll \mu \ll \eta \ll \gamma\ll 1.
\end{equation}

Applying \cref{absorbing} yields an absorbing path $P_{abs}$ such that $|V(P_{abs})| \leq 8\sqrt{\mu}n$.
We then apply \cref{reservoir} with $U = V(P_{abs})$, obtaining a reservoir set $\mathcal{R} \subseteq V(G) \setminus V(P_{abs})$ such that $|\mathcal{R}| \leq 2^{-6} \tau n$.

We next construct a cycle containing $P_{abs}$ and leaving at most $\mu n$ vertices of $G$ uncovered.
  The remaining vertices will later be absorbed by means of \cref{absorbing}.
    Set $G_1=G\backslash(\mathcal{R}\cup V(P_{abs}))$ and let $|G_1|=n_1$. By \eqref{1}, $\sigma_{2}(G_1)\ge\sigma_{2}(G)-2(|V(P_{abs})|+|\mathcal{R}|)\ge(1+\eta/2)n_1.$
    Apply \cref{lem:almost-covering} with parameters \(\xi\) and $\eta/2$ to obtain a family \(\mathcal{P} := \{P_1, P_2, \ldots, P_t\}\) of vertex-disjoint paths with \(t \le K\), along with a set of uncovered vertices of size at most \( \xi n_1\).

    Now \cref{reservoir} allows us to connect the dominant direction of each $P_i\in \mathcal{P}$ and $P_{abs}$. Define \(P_0 := P_{abs}\), and let \(u_0, v_0\) be the end vertices of \(P_{abs}\) given by \cref{absorbing}. 
    For each \(1 \le i \le t\), label the two end vertices of $P_i$ as $u_i$ and $v_i$ so that, when $P_i$ is traversed from $u_i$ to $v_i$, at least $\sigma_{\max}(P_i)$ of its edges are oriented consistently with this traversal.
 Consider first \( P_0 \) and \( P_1 \). Either $v_0$ is adjacent to $u_1$, or, by \cref{reservoir}, there exists a \( (v_0, u_1) \)-path with one internal vertex lying in the reservoir set \( \mathcal{R} \).
We may apply \cref{reservoir} greedily to construct a $(v_i,u_{i+1})$-path connecting $P_i$ and $P_{i+1}$ for each $0\leq i\leq t$, where $u_{t+1}:=u_0$, since, by \eqref{1}, at most $K+1\leq 2^{-11}\tau^2 n$ vertices of $\mathcal{R}$ are used.

Now, let \(U := V(G) \setminus V(C_0)\). 
As $|U| \le |\mathcal{R}|+\xi n_1\le 2^{-6}\tau n+\xi n\le\mu n$, \cref{absorbing} yields a path $P$ with $V(P)=V(P_{abs})\cup U$ and the same vertices as $P_{abs}$. Replacing \( P_{abs} \) with \( P \) in \( C_0 \), we obtain the desired  Hamilton cycle $C$.

It remains to verify the bound on $\sigma_{\max}(C)$.
Traverse \(C_0\) so that, for each \(i\in[t]\), the path \(P_i\) is traversed from \(u_i\) to \(v_i\).
By the choice of \(u_i,v_i\), the path \(P_i\) contributes at least \(\sigma_{\max}(P_i)\) edges oriented consistently with this traversal.
 Neither the connection step nor the absorption step changes any edge of the paths in $\mathcal P$. Therefore, $\sigma_{\max}(C) \geq \sigma_{\max}(\mathcal{P}).$
Finally, by \cref{lem:almost-covering} and \eqref{1}, we  obtain a Hamilton cycle $C$ satisfying
\begin{equation*}
    \sigma_{\max}(C)\ge \sigma_{\max}(\mathcal{P}) \geq \frac{\sigma_2(G_1)}{2} - \xi n_1 \geq \frac{1}{2}\left( \sigma_2(G) - 2(2^{-6}\tau n+8\sqrt{\mu}n) \right) - \xi n \geq \frac{\sigma_2(G)}{2} - \gamma n.
\end{equation*}
Since every oriented Hamilton cycle trivially satisfies $\sigma_{\max}(C)\ge n/2$, we conclude that
$\sigma_{\max}(C)\ge
\max\{n/2,\sigma_2(G)/2-\gamma n\}$,
completing the proof.

\section{Concluding remarks}\label{sec:remark}
In this paper, we use the absorption method to obtain an asymptotic affirmative answer to \cref{pro1} of Ai et al. It would be interesting to determine whether the structural approach developed in \cite{Freschi} can be adapted to the Ore-type setting and, in particular, whether it can lead to an exact version of \cref{th2}.

Since cycle factors generalize Hamilton cycles, it is natural to ask whether a corresponding oriented-discrepancy statement holds for prescribed cycle factors. This leads to the following oriented-discrepancy analogue of the El-Zahar conjecture \cite{El}.

\begin{conjecture}\label{conj:factor}
Let $G$ be an oriented graph on $n = n_1 + n_2 + \cdots + n_k$ vertices, where $n_i \ge 3$ for each $1 \le i \le k$. If $\delta(G) \ge \sum_{i=1}^k \lceil n_i/2 \rceil$, then $G$ contains a cycle factor $\mathcal{C}=\{C_1,C_2,\dots,C_k\}$ satisfying $\sigma_{\max}(\mathcal{C}) \ge \delta(G)-(k-1)$ with  $|C_i| = n_i$.
\end{conjecture}

The bound in \cref{conj:factor}, if true, would be best possible. Indeed, let \(G\) be a transitive tournament on \(n\) vertices. Then \(\delta(G)=n-1\). Since a transitive tournament contains no directed cycle, every oriented cycle \(C_i\) in \(G\) satisfies $\sigma_{\max}(C_i)\le |C_i|-1.$
Consequently, every cycle factor \(\mathcal C=\{C_1,\ldots,C_k\}\) with \(|C_i|=n_i\) satisfies $\sigma_{\max}(\mathcal C)\le \sum_{i=1}^k(n_i-1)=n-k=\delta(G)-(k-1).$

\vspace{0.5cm}
 \noindent \textbf{\large Acknowledgment.}  We are grateful to   V\'aclav Rozho\v n and Robert Samal for correcting the right statement of \cref{conj:factor}.

\end{document}